\documentclass[12pt,reqno]{amsart}
\usepackage{enumerate,amsmath,xpatch,relsize,graphics,graphicx,fontenc}
\usepackage{mathtools}

\xapptocmd\normalsize{%
	\abovedisplayskip=12pt plus 3pt minus 9pt
	\abovedisplayshortskip=0pt plus 3pt
	\belowdisplayskip=12pt plus 3pt minus 9pt
	\belowdisplayshortskip=7pt plus 3pt minus 4pt
}{}{}
\theoremstyle{definition}
\newtheorem{definition}{Definition}[section]
\newtheorem{remark}[definition]{Remark}

\newtheorem{illustration}[definition]{Illustration}
\theoremstyle{plain}
\newtheorem{theorem}[definition]{Theorem}

\newtheorem{lemma}[definition]{Lemma}

\numberwithin{equation}{section}

\title[Estimates for initial coefficients of certain Bi--univalent functions]{Estimates for initial coefficients of \\ certain Bi--univalent functions}

\author[V. Madaan]{Vibha Madaan}
\address{Department of Mathematics, University of Delhi, Delhi--110 007, India}
\email{vibhamadaan47@gmail.com}

\author[A. Kumar]{Ajay Kumar}
\address{Department of Mathematics, University of Delhi, Delhi--110 007, India}
\email{ak7028581@gmail.com}

\author[V. Ravichandran]{V. Ravichandran}
\address{Department of Mathematics, National Institute of Technology, Tiruchirappalli--620015, India}
\email{vravi68@gmail.com, ravic@nitt.edu}



\thanks{The first author is supported by Senior Research Fellowship from University Grants Commission, New Delhi, Ref. No.:1069/(CSIR-UGC NET DEC, 2016).}

\allowdisplaybreaks

\begin{document}
\maketitle
\begin{abstract}
Estimates are obtained for the initial coefficients of a normalized analytic function $f$ in the unit disk $\mathbb{D}$ such that $f$ and the analytic extension of $f^{-1}$ to $\mathbb{D}$ belong to certain subclasses of univalent functions. The bounds obtained improve some existing known bounds.
\end{abstract}
\section{Introduction and Preliminaries}\label{intro}
Let $\mathcal{A}$ be the class of analytic functions defined on the unit disk $\mathbb{D}:=\{z:|z|<1\}$ of the form
\begin{equation}\label{expr f}
f(z)=z+\sum_{n=2}^\infty a_nz^n.
\end{equation}
Suppose that $\mathcal{S}$ is the subclass of $\mathcal{A}$ consisting of univalent functions. Being univalent, the functions in class $\mathcal{S}$ are invertible; however, the inverse need not be defined on the entire unit disk. The Koebe one-quarter theorem ensures that the image of the unit disk under every univalent function contains a disk of radius $1/4$. Thus, a function $f\in\mathcal{S}$ has an inverse defined on a disk containing disk $|z|<1/4$. It can be easily seen that  \begin{equation*}\label{inverse f}
f^{-1}(w)=w-a_2w^2+(2a_2^2-a_3)w^3-(5a_2^2-5a_2a_3+a_4)w^4+\cdots
\end{equation*} in some disk of radius at least $1/4$. A function $f\in\mathcal{A}$ is said to be bi--univalent in $\mathbb{D}$ if both $f$ and analytic extension of $f^{-1}$ to $\mathbb{D}$ are univalent in $\mathbb{D}$. The class of bi--univalent functions, denoted by $\sigma$, was introduced by Lewin \cite{MR0206255} in 1967, who also showed that the second coefficient of a bi--univalent function satisfies the inequality $|a_2|\leq 1.51$. Let $\sigma_1$ be the class of the functions $f=\phi\circ\psi^{-1}$, where $\phi$ and $\psi$ are univalent analytic functions mapping $\mathbb{D}$ onto a domain containing $\mathbb{D}$ and satisfy $\phi'(0)=\psi'(0)$. Clearly, $\sigma_1\subset \sigma$, though $\sigma_1\neq \sigma$ (see \cite{MR0604686}). In 1969, Suffridge \cite{MR0240297} gave a function in class $\sigma_1$ with $a_2=4/3$ and conjectured that $|a_2|\leq 4/3$ for the functions in the class $\sigma$. Netanyahu \cite{MR0235110} proved the conjecture for the subclass $\sigma_1$. The conjecture was later disproved by Styler and Wright \cite{MR0609659} in 1981, who showed that $a_2>4/3$ for some function in $\sigma$. Brannan and Clunie \cite{MR0623462} conjectured that $|a_2|\leq \sqrt{2}$ for a function $f\in\sigma$. Kedzierawski \cite{MR0997757} proved this for a special case when the functions $f$ and $f^{-1}$ are starlike functions.

For analytic functions $f$ and $g$ in $\mathbb{D}$, the function $f$ is \emph{subordinate} to the function $g$, written as $f(z)\prec g(z)$, if there is a Schwarz function $w$  such that $f=g\circ w$. If $g$ is univalent, then $f(z)\prec g(z)$ if and only if $f(0)=g(0)$ and $f(\mathbb{D})\subseteq g(\mathbb{D})$. The method of subordination is quite useful for establishing relations in terms of inequalities in the complex plane. Padmanabhan and Parvatham \cite{MR0819562} gave a unified representation of various classes of starlike and convex functions using convolution with the function $z/(1-z)^\alpha$, for $\alpha\in\mathbb{R}$. Later in 1989, for a convex function $h$ and a fixed function $g$, Shanmugam \cite{MR0994916} introduced a class $\mathcal{S}^*_g(h)$ which consists of functions $f\in\mathcal{A}$ satisfying $z(f*g)'(z)/(f*g)(z)\prec h(z)$. Further, if $g(z)=z/(1-z)$ and $h=\varphi$ is an analytic function with positive real part in $\mathbb{D}$ such that $\varphi(0)=1$, $\varphi'(0)>0$ and $\varphi(\mathbb{D})$ is symmetric about the real axis and starlike with respect to $1$, then the class $\mathcal{S}^*_g(h)$ reduces to the class $\mathcal{S}^*(\varphi)$ which was introduced by Ma and Minda \cite{MR1343506}. The growth, distortion and covering theorems for the class $\mathcal{S}^*(\varphi)$ are also proved in \cite{MR1343506}. For particular choices of $\varphi$, we have the following subclasses of the univalent functions. If $\varphi(z)=(1+Az)/(1+Bz)$, where $-1\leq B<A\leq 1$, then the class $\mathcal{S}^*(\varphi)$ is termed as the class of \emph{Janowski starlike functions} \cite{MR0267103}, denoted by $\mathcal{S}^*[A,B]$. For $0\leq \beta<1$, the class $\mathcal{S}^*[1-2\beta,-1]=:\mathcal{S}^*(\beta)$ is the class of \emph{starlike functions of order $\beta$} and for $\beta = 0$, the class $\mathcal{S}^*:=\mathcal{S}^*(0)$ is simply the class of \emph{starlike functions}. If $0<\alpha\leq 1$, then the class $\mathcal{SS}^* (\alpha) :=\mathcal{S}^* (((1+z)/(1-z))^\alpha)$ is the class of \emph{strongly starlike functions of order $\alpha$}. Similarly, the class $\mathcal{K}(\varphi)$ of convex functions consists of the univalent functions satisfying $1+zf''(z)/f'(z)\prec\varphi(z)$. Let $\mathcal{R}(\varphi)$ be the class of univalent functions satisfying $f'(z)\prec\varphi(z)$. For $b\in\mathbb{C}\setminus\{0\}$ and $p\in\mathbb{N}$, the classes $\mathcal{R}_{b,p}(\varphi)$ and $\mathcal{S}^*_{b,p}(\varphi)$ consist of the functions of the form $f(z)=z^p+\sum_{n=p+1}^\infty a_nz^n$ satisfying \[ 1+\frac{1}{b}\left(\frac{f'(z)}{pz^{p-1}}-1\right) \prec \varphi(z) \quad \text{and} \quad  1+\frac{1}{b} \left(\frac{1}{p} \frac{zf'(z)}{f(z)} -1\right) \prec \varphi(z),  \] respectively. Ali \emph{et al.\@} \cite{MR2323552} obtained Fekete-Szeg\"{o} inequalities and bound on the coefficient $a_{p+3}$ for the functions in these classes. On coefficient estimates for the functions that belong to certain subclasses of univalent functions, one can refer \cite{MR0232926, MR2020565}.

Analogous to the class of starlike (and convex) functions of order $\beta$ (with $0\leq \beta<1$), the class of \emph{bi--starlike} (and \emph{bi-convex}) \emph{functions of order $\beta$}, denoted by $\mathcal{S}^*_\sigma(\beta)$ (and $\mathcal{K}_\sigma(\beta)$), is the class of bi--univalent functions $f$ such that $f$ and analytic extension of $f^{-1}$ to $\mathbb{D}$ are both starlike (and convex) of order $\beta$ in $\mathbb{D}$. For $0<\alpha\leq 1$, a bi--univalent function $f$ is in class $\mathcal{S}^*_\sigma[\alpha]$ of \emph{strongly bi--starlike functions of order $\alpha$} if $f$ and analytic extension of $f^{-1}$ to $\mathbb{D}$ are strongly starlike functions of order $\alpha$ in $\mathbb{D}$. Brannan and Taha \cite{MR0951657} introduced these classes and gave bound on initial coefficients of the functions in these classes. Also, for a function $f\in\mathcal{K}_\sigma(0)$ given by \eqref{expr f}, they showed $|a_2|\leq 1$ and $|a_3|\leq 1$ with extremal function given by $z/(1-z)$ and its rotations. Particularly if $\beta=0$, then the class $\mathcal{S}^*_\sigma(\beta)$ reduces to the class of bi--starlike functions. Kedzierawski \cite{MR0997757} proved that for a bi--starlike function $f$ of the form \eqref{expr f}, $|a_2|\leq \sqrt{2}$. Further, \cite{MR3450961} and \cite{MR3178537} improved the estimates for coefficients $a_2$ and $a_3$ and also found estimates for the fourth coefficient for the functions in classes $\mathcal{S}^*_\sigma(\beta)$ and $\mathcal{S}^*_\sigma[\alpha]$. For coefficient estimates for the functions in some particular subclasses of bi--univalent functions, one may see \cite{MR2855984, MR2803711, SKR, MR3064526, MR2665593, MR2902367, MR2943990, MR3214422}.

Let the function $\varphi$ be an analytic function in $\mathbb{D}$ of the form  \begin{equation}\label{expr phi} \varphi(z)= 1+B_1z+B_2z^2+B_3z^3+\cdots, \end{equation} where $B_1>0$. For the function $\varphi$ and $\lambda\geq 0$, Kumar \emph{et al.\@} \cite{SKR} introduced the following subclass $\mathcal{R}_\sigma(\lambda,\varphi)$ of bi--univalent functions.
\begin{definition}\label{definition R}
Let $\lambda\geq 0$. A bi--univalent function $f$ given by \eqref{expr f} is in class $\mathcal{R}_\sigma(\lambda,\varphi)$, if it satisfies \[ (1-\lambda)\frac{f(z)}{z}+\lambda f'(z)\prec \varphi(z) \quad{}\text{and}\quad{} (1-\lambda)\frac{g(w)}{w}+\lambda g'(w)\prec \varphi(w), \] where $g$ denotes the univalent extension of $f^{-1}$ to the unit disk.
\end{definition}
With the particular values of $\lambda$ and $\varphi$, the class $\mathcal{R}_\sigma(\lambda,\varphi)$ reduces to many earlier classes as mentioned below:
\begin{enumerate}[(i)]
\item $\mathcal{R}_\sigma(\lambda,(1+(1-2\beta)z)/(1-z))=\mathcal{R}_\sigma (\lambda,\beta)\quad (\lambda\geq 1; 0\leq \beta<1)$ \cite[Definition 3.1]{MR2803711}
\item $\mathcal{R}_\sigma(\lambda,((1+z)/(1-z))^\alpha)=\mathcal{R}_{\sigma,\alpha} (\lambda)\quad (\lambda\geq 1; 0< \alpha\leq 1)$ \cite[Definition 2.1]{MR2803711}
\item $\mathcal{R}_\sigma(1,\varphi)=\mathcal{R}_\sigma(\varphi)$ \cite[p. 345]{MR2855984}.
\item $\mathcal{R}_\sigma(1,(1+(1-2\beta)z)/(1-z))=\mathcal{R}_\sigma (\beta)\quad (0\leq \beta<1)$ \cite[Definition 2]{MR2665593}.
\item $\mathcal{R}_\sigma(1,((1+z)/(1-z))^\alpha)=\mathcal{R}_{\sigma,\alpha}\quad (0< \alpha\leq 1)$ \cite[Definition 1]{MR2665593}
\end{enumerate}

The class of bi--starlike functions of Ma-Minda type was given by Ali \emph{et al.\@} \cite{MR2855984}.
\begin{definition}\label{defn bistarlike}
A function $f\in\sigma$ of the form \eqref{expr f}, is said to be in the class of \emph{Ma-Minda bi--starlike functions}, denoted by $\mathcal{S}^*_\sigma(\varphi)$, if the following subordinations hold: \[ \frac{zf'(z)}{f(z)}\prec \varphi(z) \quad\text{and}\quad \frac{wg'(w)}{g(w)}\prec\varphi(w),  \] where $g$ denotes the univalent extension of $f^{-1}$ to $\mathbb{D}$ and $\varphi$ is the function of the form \eqref{expr phi} satisfying the conditions as in the definition of the class  $\mathcal{S}^*(\varphi)$ as mentioned earlier.
\end{definition}
The class $\mathcal{S}^*_\sigma(\varphi)$ includes some well-known classes of the bi--univalent functions. For example:
\begin{enumerate}[(i)]
\item $\mathcal{S}^*_\sigma((1+(1-2\beta)z)/(1-z))=:\mathcal{S}^*_\sigma(\beta)$, $0\leq \beta<1$.
\item $\mathcal{S}^*_\sigma(((1+z)/(1-z))^\alpha)=:\mathcal{S}^*_\sigma[\alpha]$, $0<\alpha\leq 1$.
\end{enumerate}

Using the Fekete-Szeg\"{o} inequalities and principles of subordination, in this paper, the estimates for the coefficients $a_2$ and $a_3$ of the functions of the form \eqref{expr f} in the classes $\mathcal{R}_\sigma(\lambda,\varphi)$ and $\mathcal{S}^*_\sigma(\varphi)$ have been obtained. Moreover, the estimates so obtained are observed to be an improvement over the ones derived in \cite{MR2855984, Nisha, SKR}. For some particular choices of $\lambda$ and $\varphi$, the bounds determined are smaller than those mentioned in \cite{MR0951657, MR2803711, MR3450961, MR3178537, MR2665593} for the coefficients of the functions in the respective classes.

More precisely, the following theorem derives the estimates for the coefficients $a_2$ and $a_3$ for the functions given by \eqref{expr f} that belong to the class $\mathcal{R}_\sigma(\lambda,\varphi)$.
\begin{theorem}\label{bound theorem}
Let $\varphi$ be an analytic function given by the series \eqref{expr phi} such that $B_2\in\mathbb{R}$. For $\lambda\geq 0$, let the function $f\in\mathcal{R}_\sigma (\lambda,\varphi)$ and $\tau:=(1+\lambda)^2/(1+2\lambda)$.
\begin{enumerate}[(a)]
\item\label{Ra} If $\tau B_2\leq B_1^2$, then \[ |a_2|\leq \frac{B_1 \sqrt{B_1}}{ \sqrt{(1+2\lambda) (B_1^2-\tau B_2+\tau B_1)}} \quad \text{and}\quad |a_3|\leq \frac{B_1}{1+2\lambda}\max\left\{\frac{B_1^2}{B_1^2-\tau B_2+\tau B_1}, 1  \right\}.  \]
\item\label{Rb} If $\tau B_2\geq B_1^2$, then \[ |a_2|\leq \frac{B_1\sqrt{B_1}}{\sqrt{(1+2\lambda)(\tau B_2+\tau B_1-B_1^2)}} \quad \text{and} \quad |a_3|\leq \frac{B_1}{1+2\lambda}\max\left\{\frac{B_1^2}{\tau B_2+\tau B_1-B_1^2}, 1  \right\}.  \]
\end{enumerate}
\end{theorem}
\begin{remark}\label{remark R}
Theorem \ref{bound theorem} is an improvement over the coefficient estimates obtained in \cite{MR2855984,MR2803711,SKR,MR2665593}. For an analytic function $\varphi$ of the form \eqref{expr phi}, Kumar \emph{et al.\@} \cite[Theorem 2.2]{SKR} obtained a bound on the coefficient $a_2$ of the function $f\in\mathcal{R}_\sigma(\lambda,\varphi)$ for $\lambda\geq 0$. In addition, if $B_2\in\mathbb{R}$, by the means of the following comparisons, it may be noted that  Theorem \ref{bound theorem} gives an estimate for $a_2$ which is smaller than the one given by \cite[Theorem 2.2]{SKR}. We can see that if $\tau B_2\leq B_1^2$ and $B_2\leq B_1$, then  \[ \frac{2B_1-B_2}{B_1}- \frac{B_1^2}{B_1^2-\tau B_2+\tau B_1}= \frac{(B_1-B_2)(2\tau B_1+B_1^2-\tau B_2)}{B_1(B_1^2-\tau B_2+\tau B_1)} \geq 0. \]
Therefore, $\min\left\{\frac{2B_1-B_2}{B_1}, \frac{B_1^2}{B_1^2-\tau B_2+\tau B_1}\right\}= \frac{B_1^2}{B_1^2-\tau B_2+\tau B_1}$ which implies that the estimate obtained for $a_2$ using Theorem \ref{bound theorem} for this case is less than $\sqrt{(2B_1-B_2)/(1+2\lambda)}$. Similarly, if $\tau B_2\leq B_1^2$ and $B_2\geq B_1$, then \[ \frac{B_2}{B_1}-\frac{B_1^2}{B_1^2 -\tau B_2+\tau B_1} =\frac{(B_1^2 -\tau B_2)(B_2-B_1)}{B_1(B_1^2-\tau B_2+\tau B_1)} \geq 0.\]
Next, the case when the conditions $\tau B_2\geq B_1^2$ and $B_2\leq B_1$ hold, it follows that \[ \frac{2B_1-B_2}{B_1}-\frac{B_1^2}{\tau B_2+\tau B_1-B_1^2}\geq  \frac{B_1(\tau -B_1)^2}{\tau (\tau B_2+\tau B_1-B_1^2)}\geq 0\] and further, if $\tau B_2\geq B_1^2$ and $B_2\geq B_1$, then the inequality \[  \frac{B_2}{B_1}-\frac{B_1^2}{\tau B_2+\tau B_1-B_1^2}=\frac{(\tau B_2-B_1^2)(B_2+B_1)}{B_1(\tau B_2+\tau B_1-B_1^2)}\geq 0\] holds.
	
Now let us consider the class $\mathcal{R}_{\sigma,\alpha}(\lambda):=\mathcal{R}_\sigma(\lambda,((1+z)/(1-z))^\alpha)$ for $0<\alpha\leq 1$.  Clearly, $B_1=2\alpha$ and $B_2=2\alpha^2$. For a function $f$ given by \eqref{expr f} in the class $\mathcal{R}_{\sigma,\alpha}(\lambda)$, Theorem \ref{bound theorem} yields \[ |a_2|\leq\begin{cases}
\frac{2\alpha}{\sqrt{(1+\lambda)^2+\alpha(1-\lambda^2+2\lambda)}} &\quad\text{if}\quad 1\leq \lambda \leq 1+\sqrt{2}\\
\frac{2\alpha}{\sqrt{(1+\lambda)^2-\alpha(1-\lambda^2+2\lambda)}} &\quad\text{if}\quad \lambda\geq 1+\sqrt{2}.
\end{cases}   \] It can be verified that if $1\leq \lambda\leq 1+\sqrt{2}$, then the bound derived for  $a_2$ coincides with that obtained by Frasin and Aouf \cite[Theorem 2.2]{MR2803711}, whereas the estimate obtained for $a_2$ for the part $\lambda\geq 1+\sqrt{2}$ is smaller than that in \cite[Theorem 2.2]{MR2803711}. Likewise, using Theorem \ref{bound theorem}, we can see that  $|a_3|\leq 2\alpha/(1+2\lambda)$ which is less than the bound for $a_3$ derived in \cite[Theorem 2.2]{MR2803711}.
	
Similarly, let $\varphi(z)=(1+(1-2\beta)z)/(1-z)$ for $0\leq \beta<1$. As a result of Theorem \ref{bound theorem}, the functions in the class $\mathcal{R}_\sigma(\lambda,\beta)$ satisfy  \[ |a_2|\leq \begin{cases}
\sqrt{\frac{2(1-\beta)}{1+2\lambda}} &\quad\text{if}\quad 0\leq \beta\leq \frac{1-\lambda^2+2\lambda}{2(1+2\lambda)}\\
(1-\beta)\sqrt{\frac{2}{\lambda^2+\beta(1+2\lambda)}} &\quad\text{if}\quad \frac{1-\lambda^2+2\lambda}{2(1+2\lambda)}\leq \beta<1
\end{cases}  \] and $|a_3|\leq 2(1-\beta)/(1+2\lambda)$. Again, the estimate for $a_2$ so determined for the part when $0\leq\beta\leq (1-\lambda^2+2\lambda)/(2(1+2\lambda))$ is same as that obtained by Frasin and Aouf \cite[Theorem 3.2]{MR2803711}. For $(1-\lambda^2+2\lambda)/(2(1+2\lambda)) \leq \beta< 1$, the estimate for $a_2$, derived using Theorem \ref{bound theorem}, is refined in comparison with \cite[Theorem 3.2]{MR2803711}. The estimate for the coefficient $a_3$ obtained using Theorem \ref{bound theorem} is smaller than the one in \cite[Theorem 3.2]{MR2803711}. Moreover, the coefficient estimates derived above for the functions in classes $\mathcal{R}_{\sigma,\alpha} (\lambda)$  and $\mathcal{R}_\sigma(\lambda,\beta)$ are valid for $\lambda \geq 0$.

Also, Ali \emph{et al.\@} \cite[Theorem 2.1]{MR2855984} derived bound on the coefficients $a_2$ and $a_3$ of a function $f \in \mathcal{R}_\sigma (\varphi)$ of the form \eqref{expr f}. It may be noted that the estimates for the coefficients $a_2$ and $a_3$ of the function $f\in\mathcal{R}_\sigma(\varphi)$ given using Theorem \ref{bound theorem} improve the estimates given in \cite[Theorem 2.1]{MR2855984} provided $\varphi''(0)\in\mathbb{R}$.

Furthermore, the coefficient estimates for the functions in the classes $\mathcal{R}_{\sigma,\alpha}$ and $\mathcal{R}_\sigma(\beta)$ determined in \cite[Theorem 1]{MR2665593} and \cite[Theorem 2]{MR2665593}, respectively are particular cases for the above-mentioned estimates.
\end{remark}
The next theorem determines the estimates for the initial coefficients for a function in the class $\mathcal{S}^*_\sigma(\varphi)$.
\begin{theorem}\label{bound theorem starlike}
Let $f\in\mathcal{S}^*_\sigma(\varphi)$, where $\varphi''(0)\in\mathbb{R}$.
\begin{enumerate}[(a)]
\item\label{Sa} If $B_2\leq B_1^2$, then \[ |a_2|\leq \frac{B_1\sqrt{B_1}}{\sqrt{B_1^2+B_1-B_2}} \quad \text{and}\quad |a_3|\leq \max\left\{ \frac{B_1^3}{B_1^2-B_2+B_1}, \frac{B_1}{2}  \right\}. \]
\item\label{Sb} If $B_2\leq B_1^2$, then \[ |a_2|\leq\frac{B_1\sqrt{B_1}}{\sqrt{B_2+B_1-B_1^2}}\quad \text{and} \quad |a_3|\leq \max\left\{ \frac{B_1^3}{B_2+B_1-B_1^2}, \frac{B_1}{2}  \right\}. \]
\end{enumerate}
\end{theorem}
\begin{remark}
Bohra \emph{et al.\@} \cite[Corollary 2.3]{Nisha} and Ali \emph{et al.\@} \cite[Corollary 2.1]{MR2855984} gave estimates on the coefficients $a_2$ and $a_3$ of the functions in the class $\mathcal{S}^*_\sigma(\varphi)$. In addition, let us assume that $B_2\in\mathbb{R}$. By means of inequalities similar to those in Remark \ref{remark R}, we can see that the estimates for the coefficients $a_2$ and $a_3$ of a function in the class $\mathcal{S}^*_\sigma(\varphi)$, obtained using Theorem \ref{bound theorem starlike}, improve those derived in the above references.

Particularly if $\varphi(z)=((1+z)/(1-z))^\alpha$, $(0<\alpha\leq 1)$, the Theorem \ref{bound theorem starlike} readily yields that for a function $f\in\mathcal{S}^*_\sigma[\alpha]$ of the form \eqref{expr f}, we have $|a_2|\leq 2\alpha/(\sqrt{\alpha+1})$, while \[ |a_3|\leq \alpha \quad \text{if}\quad 0<\alpha\leq 1/3, \quad \text{and} \quad |a_3|\leq 4\alpha^2/(\alpha+1) \quad \text{if}\quad 1/3\leq \alpha\leq 1 \] which coincides with the estimates for $a_3$ as mentioned in \cite[Theorem 2.1]{MR3178537}. For a function $f\in\mathcal{S}^*_\sigma(\beta)$ $(0\leq \beta<1)$, a bi--starlike function of order $\beta$, using Theorem \ref{bound theorem starlike}, we may solve to get \[ |a_2|\leq \sqrt{2(1-\beta)} \quad \text{if} \quad 0\leq \beta\leq 1/2, \quad \text{whereas}\quad |a_2|\leq (1-\beta)\sqrt{2/\beta} \quad \text{if}\quad 1/2\leq \beta<1. \]  Further, \[ |a_3|\leq \begin{cases}
2(1-\beta) &\quad \text{if} \quad 0\leq \beta\leq 1/2\\
2(1-\beta)^2/\beta &\quad \text{if}\quad 1/2\leq \beta \leq 2/3\\
1-\beta &\quad \text{if}\quad 2/3\leq \beta<1.
\end{cases}  \] The bounds for $a_2$ and $a_3$ obtained above are smaller than those given by \cite{MR3450961}. Also it can be seen that the bounds obtained as a result of Theorem \ref{bound theorem starlike} are an improvement over the ones given by Brannan and Taha \cite{MR0951657}.
\end{remark}
\section{Proofs of the main results}
We now prove following lemmas which are useful to prove the main result.
\begin{lemma}\label{lemma a2}
Let $\xi\in\mathbb{R}$, $\eta>0$. Let the function $G\colon\mathbb{R}\to[0,\infty)$ be defined by \[ G(x):=\max\{1,|\eta x-\xi| \}.  \] Then \[ \inf_{x,y\in\mathbb{R}}\frac{G(x)+G(y)}{|2-x-y|}=\begin{cases}
\frac{1}{1-\gamma} &\quad\text{if}\quad\xi\leq \eta\\
\frac{1}{\rho-1} &\quad\text{if}\quad\xi\geq \eta,
\end{cases}   \] where $\gamma:=(\xi-1)/\eta$ and $\rho:=(\xi+1)/\eta$.
\end{lemma}
\begin{proof}
The function $G(x)$ can be simplified to \[ G(x)=\begin{cases}
\xi-\eta x &\quad \text{if}\quad x\leq \gamma\\
1 &\quad\text{if} \quad \gamma \leq x\leq \rho\\
\eta x-\xi &\quad \text{if} \quad x \geq \rho.
\end{cases}  \] Let $H$ be a function on $\mathbb{R}^2 \setminus \{(x,y): x+y=2\}$ defined by   \[ H(x,y):=\frac{G(x)+G(y)}{|2-x-y|}.    \] It may be noted that $\lim_{x+y\rightarrow 2}H(x,y)=\infty$. Being a non-negative real-valued function, $H$ has a non-negative infimum in $\mathbb{R}$. To find the infimum of the function $H$, we consider the possibilities as $x$ and $y$ vary in the domain of definition. Since $H$ is a symmetric function in the variables $x$ and $y$, the infimum of the function $H$ if $\gamma\leq x\leq \rho$ and $y\leq \gamma$ coincides with that whenever $x\leq \gamma$ and $\gamma\leq y\leq \rho$. Similarly, the conditions $x\geq \rho$, $y\leq \gamma$, and $\gamma\leq x\leq \rho$, $y\geq \rho$ result into the same infimum of the function $H$ as given by $x\leq \gamma$, $y\geq\rho$, and $x\geq \rho$, $\gamma\leq y\leq \rho$, respectively. Therefore, the process of determining the infimum of the function $H$ reduces to finding infimum of $H$ in the following cases:

\textbf{Case 1: } $x,y\leq \gamma$. \\
The function $H$ becomes \[ H(x,y)=\frac{1}{|2-x-y|}(2\xi-(x+y)\eta).  \] This case may be divided into three subcases \emph{viz.\@} $\xi\leq \eta$, $\eta < \xi\leq \eta+1$ and $\xi> \eta+1$. If $\xi\leq \eta$, it is clear that $\gamma<1$. Hence, the function $H$ reduces to \[ H(x,y)= \frac{1}{2-x-y}(2\xi-(x+y)\eta).  \]  The function $H$ has no critical points in the set $(-\infty,\gamma)\times(-\infty,\gamma)$ which states that $H$ does not attain its minimum in its set. Thus, the minimum of the function $H$, if it exists, is attained on the boundary; however, $H$ may have its infimum as $x$ or $y$ approach $-\infty$ or both. Along the line $y=\gamma$, it is easy to see that the function $H(x,\gamma)$ is decreasing in $x$ and $\min H(x,\gamma)=H(\gamma, \gamma)= 1/(1-\gamma)$. Similarly, $\min H(\gamma,y)= 1/(1-\gamma)$. At the same time \[ \lim_{x\rightarrow-\infty} H(x,y) =\lim_{y\rightarrow-\infty} H(x,y)=\eta.\] For $\xi\leq \eta$, since $\min\{ \eta, 1/(1-\gamma) \}=1/(1-\gamma)$, we conclude that \[\inf H(x,y)=\min H(x,y)= H(\gamma, \gamma)=1/(1-\gamma). \]

Let us now assume that $\eta< \xi\leq \eta+1$ which implies $\gamma\leq 1$. Being increasing functions in $x$ and $y$, $H(x,\gamma)$ and $H(\gamma,y)$, respectively, do not attain their minimum; however, we have \[\inf H(x,y) =\lim_{x\rightarrow-\infty} H(x,y) =\lim_{y\rightarrow-\infty} H(x,y) =\eta.\] Therefore, in this case, $ \inf H(x,y)=\eta$.

Suppose $\xi> \eta+1$ which means $\gamma>1$. Whenever $x+y<2$, by virtue of the subcase $\eta< \xi\leq \eta+1$, we observe that $\inf H(x,y)=\eta$, whereas, if $x+y>2$, then the function $H$ is given by \[ H(x,y)=\frac{1}{x+y-2}(2\xi-(x+y)\eta).  \] The function $H$ attains its minimum along the edges $x=\gamma$ and $y=\gamma$. It can be verified that the functions $H(x,\gamma)$ and $H(\gamma,y)$ both attain their minimum at the point $(\gamma,\gamma)$ with the minimum value $H(\gamma,\gamma)=1/(\gamma-1)$ and on choosing the least amongst the values $\eta$ and $1/(\gamma-1)$, we infer that \[ \inf H(x,y)= \begin{cases}
\eta &\quad\text{if}\quad \eta+1< \xi\leq \eta+2\\
1/(\gamma-1) &\quad\text{otherwise}
\end{cases} \] provided $\xi> \eta+1$. In view of the observations made above in each of the subcase, by selecting the least of the corresponding infimum, it follows that \[ \inf_{x,y\leq \gamma} H(x,y)=\begin{cases}
\eta &\quad\text{if}\quad \eta\leq \xi\leq \eta+2\\
1/|1-\gamma| &\quad\text{otherwise.}
\end{cases}   \]

\textbf{Case 2: }$x\leq \gamma$ and $\gamma\leq y\leq \rho$.\\ In this case, the function $H$ becomes \[ H(x,y) =\frac{1}{|2-x-y|}(\xi+1-x\eta) =\frac{\eta}{|2-x-y|}(\rho-x).   \] It may be noted that the function $H$ either attains its minimum on the boundary or has its infimum as $x$ approaches $-\infty$.

This case may be partitioned into the subcases $\xi\leq \eta-1$, $\eta-1<\xi\leq \eta$ and $\xi>\eta$. Whenever $\xi\leq \eta-1$, it can be seen that the functions $H(x,\gamma)$ and $H(x,\rho)$, being decreasing functions of $x$, attain their minimum at the points $(\gamma,\gamma)$ and $(\gamma,\rho)$, respectively with the minimum values $1/(1-\gamma)$ and $2/(2-\gamma-\rho)$, respectively.  Also, $H(\gamma,y)\geq 1/(1-\gamma)$ for $\gamma \leq y\leq \rho$. Since $y$ can be a finite real number, the function $H$ may have its infimum as $x$ approaches $ -\infty$. Clearly, $\lim_{x\rightarrow -\infty} H(x,y)=\eta$. Since the least of the derived values with $\xi \leq \eta-1$ is $1/(1-\gamma)$,  we conclude that $ \inf H(x,y)= \min H(x,y)= H(\gamma, \gamma)=1/(1-\gamma)$.

For the case if $\eta-1 < \xi\leq \eta$, we get that $\min H(x,\gamma)=\min H(\gamma,y)=H(\gamma, \gamma)= 1/(1-\gamma)$, whereas $H(x,\rho)$ does not attain its minimum. It is easy to verify that \[\inf H(x,\rho)=\inf H(x,y)=\lim_{x\rightarrow -\infty} H(x,y)=\eta.\] Again by simple computations, we can see that $ \inf H(x,y)=1/(1-\gamma)$.

Let us assume that $\xi>\eta$ which is same as $\gamma+\rho>2$. Whenever $x+y<2$, it may be noted that $H$ does not have a minimum value, while $\inf H(x,y)=\lim_{x\rightarrow -\infty} H(x,y)=\eta$. Further, if $x+y>2$, the function $H$ attains its minimum on the boundary. The functions $H(x, \rho)$ and $H(\gamma,y)$ have the minimum values given by $H(\gamma, \rho)=2/(\gamma+\rho-2)$. In addition, if $\xi\geq \eta+1$, then we also have $\min H(x,\gamma)=H(\gamma, \gamma)=1/(1-\gamma)$.

Choosing the least of all the values so obtained, we  deduce that \[ \inf_{x\leq \gamma, \gamma\leq y\leq \rho} H(x,y)=\begin{cases}
1/(1-\gamma) &\quad\text{if}\quad \xi\leq \eta\\
\eta &\quad\text{if}\quad \eta\leq \xi\leq \eta+1\\
2/(\gamma+\rho-2) &\quad\text{if}\quad \xi\geq \eta+1.
\end{cases}   \] Working on similar lines, infimum for the function $H$ for each of the case may be noted as follows:

\textbf{Case 3: } $x\leq \gamma$ and $y\geq\rho$.\\
This case can be viewed in four subcases as given by $\xi \leq \eta-1$, $\eta-1< \xi\leq \eta$, $\eta< \xi\leq \eta+1$ and $\xi>\eta+1$. For $\xi \leq \eta-1$, we can see that the minimum of the function $H$ is attained at the point $(\gamma, \rho)$ with the minimum value $2/(2-\gamma-\rho)$, whereas if $\eta-1< \xi\leq \eta$, then \[\inf H(x,y)=\lim_{x\rightarrow -\infty} H(x,y)=\lim_{y\rightarrow \infty}=H(x,y)= \eta.\] Solving for the other two parts similarly and selecting the least value, we get \[ \inf_{x\leq \gamma, y\geq\rho} H(x,y)=\begin{cases}
\eta &\quad\text{if}\quad \eta-1\leq \xi\leq \eta+1\\
2/(|2-\gamma-\rho|) &\quad\text{otherwise}.
\end{cases}  \]

\textbf{Case 4: } $\gamma\leq x,y\leq \rho$.\\
This case may be partitioned into the subcases \emph{viz.\@} $\xi\leq \eta-1$, $\eta-1\leq \xi\leq \eta+1$ and $\xi\geq \eta+1$. If $\xi\leq \eta-1$ which indicates $\rho\leq 1$, then $\min H(x,y)=H(\gamma, \gamma)=1/(1-\gamma)$. Assuming that $\eta-1\leq \xi\leq \eta+1$ which is same as $\gamma\leq 1\leq \rho$, we have \[ \min H(x,y)=\begin{cases}
H(\gamma, \gamma)=1/(1-\gamma) &\quad\text{if}\quad\eta-1\leq \xi\leq \eta\\
H(\rho, \rho)=1/(\rho-1) &\quad\text{if}\quad \eta\leq \xi\leq \eta+1.
\end{cases}   \] In case $\xi\geq\eta+1$ which means $\gamma\geq 1$, then $\min H(x,y)=H(\rho, \rho)=1/(\rho-1)$. Hence, we conclude that \[ \min_{\gamma\leq x,y\leq \rho} H(x,y)=\begin{cases}
1/(1-\gamma) &\quad\text{if}\quad \xi\leq \eta\\
1/(\rho-1) &\quad\text{if}\quad \xi\geq \eta.
\end{cases}   \]

\textbf{Case 5: } $x\geq \rho$ and $\gamma\leq y\leq \rho$.\\
Let $\xi\leq \eta-1$ which signifies the condition $\rho\leq 1$. In this case $\min H(x,y)=H(\gamma, \rho)=2/(2-\gamma-\rho)$. Suppose that $\eta-1\leq \xi\leq \eta$ that is $\gamma+\rho\leq 2$ and $\rho\geq 1$ for which we have $\inf H(x,y)= \lim_{x\rightarrow \infty} H(x,y)=\eta$. Given that $\xi\geq \eta$ that is $\gamma+\rho\geq 2$, $\min H(x,y)=H(\rho, \rho)=1/(\rho-1)$. Consequently, we have \[ \inf_{x\geq \rho, \gamma\leq y\leq \rho} H(x,y)=\begin{cases}
2/(2-\gamma-\rho) &\quad\text{if}\quad \xi\leq \eta-1\\
\eta &\quad\text{if}\quad \eta-1\leq\xi\leq \eta\\
1/(\rho-1) &\quad\text{if}\quad \eta\leq \xi\\
\end{cases}    \]

\textbf{Case 6: } $x,y\geq \rho$. \\
Again, the case may be divided into the subcases given by $ \xi\leq \eta-1$, $\eta-1<\xi\leq \eta $ and $\xi>\eta$. On the similar lines as followed in the case when $x,y\leq \gamma$, it can be verified that \[ \inf_{x,y\geq \rho}  H(x,y)=\begin{cases}
\eta &\quad\text{if}\quad\eta-2\leq\xi\leq \eta\\
1/|\rho-1| &\quad\text{otherwise}.
\end{cases}   \]

From the above cases, we use some simple computations to select the least value of all the infimum obtained above. Therefore, it can be concluded that \[ \inf_{x,y\in\mathbb{R}} H(x,y)=\begin{cases}
1/(1-\gamma) &\quad\text{if}\quad\xi\leq \eta\\
1/(\rho-1) &\quad\text{if}\quad\xi\geq \eta\\
\end{cases}   \] and the lemma holds.
\end{proof}
\begin{lemma}\label{lemma a3}
Let $\xi\in\mathbb{R}$ and $\eta>0$. Let the function $G\colon\mathbb{R}\to [0,\infty)$ be defined as in Lemma \ref{lemma a2}. Then \[ \inf_{x,y\in\mathbb{R}} \frac{|2-y|G(x)+|x|G(y)}{|2-x-y|}=\begin{cases}
\frac{1}{1-\gamma} &\quad\text{if}\quad 1\leq \xi\leq \eta\\
\frac{1}{\rho-1} &\quad\text{if}\quad \eta\leq \xi\leq 2\eta-1\\
1 &\quad\text{otherwise,}
\end{cases}   \] where $\gamma:=(\xi-1)/\eta$ and $\rho:=(\xi+1)/\eta$.
\end{lemma}
\begin{proof}
Let $H$ be a function on $\mathbb{R}^2\setminus \{(x,y): x+y=2\}$ defined by \[ H(x,y):= \frac{|2-y|G(x)+|x|G(y)}{|2-x-y|}. \] and $\lim_{x+y\rightarrow 2}H(x,y)=\infty$. Being a non-negative real-valued function, $H$ has a non-negative infimum. In order to obtain the infimum of the function $H$, we consider the following cases:

\textbf{Case 1: } $x,y\leq \gamma$.\\
The function $H$ becomes \[ H(x,y)=\frac{1}{|2-x-y|}\left(|2-y|(\xi-\eta x)+|x|(\xi-\eta y)\right). \] Since the function $H$ has no critical points in the set $(-\infty,\gamma)\times(-\infty,\gamma)$, it does not acquire a minimum value in this region. Thus, the function $H$ has infimum as $x$ or $y$ approach $-\infty$ or has minimum along the edges $x=\gamma$ or $y=\gamma$.

To minimize the function $H$, this case is divided into the subcases \emph{viz.\@} $\xi\leq 1$, $1\leq \xi\leq \eta+1$, $\eta+1\leq \xi\leq 2\eta+1$ and $2\eta+1\leq \xi$. If $\xi\leq 1$ which yields $\gamma\leq 0$, the function $H$ simplifies to \[ H(x,y)=\xi-\frac{2\eta x(1-y)}{2-x-y}.  \] It can easily be verified that along the edge $y=\gamma$, the function $H(x,\gamma)$ is a decreasing function of $x$, hence attains its minimum at the point $(\gamma, \gamma)$ with the minimum value $H(\gamma,\gamma)=1$ and so does the function $H(\gamma, y)$. Also, we may note that the values $\lim_{x\rightarrow -\infty}H(x,y)=\xi-2\eta(1-y)$ and $\lim_{y\rightarrow-\infty}H(x,y)=\xi-2\eta x$ exceed $1$. This implies $ \inf H(x,y)= \min H(x,y)=1  $ whenever $\xi\leq 1$.

Now if $1\leq \xi \leq \eta+1$ which means $0\leq \gamma\leq 1$, the case can be split into parts when $x\leq 0$ and when $x\geq 0$. If $x\leq 0$, the function $H(x, \gamma)$ has its minimum value given by $H(0,\gamma)=\xi$. Besides, for a fixed value of $y$, $\lim_{x\rightarrow -\infty}H(x,y)=\xi+2\eta(1-y)$, and $\lim_{y\rightarrow -\infty}H(x,y)=\xi-2\eta x$ for a particular value of $x$. We see that $\xi+2\eta(1-y)>\xi$ and $\xi-2\eta x>\xi$ for $x\leq 0$ and $y\leq \gamma$. Further, if $x\geq 0$, then the function $H$ may be expressed as \[ H(x,y)=\frac{\xi (2-y+x)-2\eta x}{2-x-y}.  \] With simple calculations, it can be seen that the function $H(x,\gamma)$ is an increasing function of $x$ provided $\eta\leq 1$. Therefore, $\min H(x,\gamma) =H(0,\gamma) =\xi$. Whereas, the condition $\eta\geq 1$ implies that $H(x,\gamma)$ is an increasing function of $x$ whenever $\xi\geq \eta$ and is a decreasing function for $1\leq \xi\leq \eta$. Hence, we infer that  \[ \min H(x,\gamma)=\begin{cases}
H(\gamma,\gamma)=1/(1-\gamma) &\quad\text{if}\quad 1\leq \xi\leq  \eta\\
H(0,\gamma)=\xi &\quad\text{if}\quad \eta\leq  \xi \leq \eta+1.
\end{cases}  \] Computing in a similar manner along the edge $x=\gamma$, it may be noted that if $\eta\leq 1$, then we have $ \inf H(\gamma,y)= \lim\limits_{y\rightarrow -\infty}H(\gamma,y)=\xi$. If $\eta \geq 1$, then the infimum of $H(x,y)$ is given by \[ \inf  H(\gamma,y) =\begin{cases}
H(\gamma,\gamma)=1/(1-\gamma) &\quad\text{if}\quad 1\leq \xi\leq \eta\\
\lim\limits_{y\rightarrow -\infty}H(\gamma,y)=\xi &\quad\text{if}\quad \eta \leq \xi\leq \eta+1.
\end{cases}   \] In addition to this, $\lim_{y\rightarrow-\infty}H(x,y)=\xi$. Consequently, for $1\leq \xi \leq \eta+1$, simplifying the values so obtained, it can be seen that if $\eta\leq 1$, then $\inf H(x,y)=\xi$, otherwise we have \[ \inf H(x,y)= \begin{cases}
1/(1-\gamma) &\quad\text{if}\quad 1\leq \xi \leq \eta\\
\xi &\quad \text{if} \quad \eta\leq \xi\leq \eta+1.
\end{cases}   \]

The subcase when we have $\eta+1\leq \xi \leq 2\eta+1$ which results in $1\leq \gamma\leq 2$, we further divide it into the parts as $x\leq 0$, $x\geq 0$ and $x+y<2$, and $x\geq 0$ and $x+y>2$. Let $x\leq 0$. The function $H(x,\gamma)$ is increasing, and has the infimum as $x$ approaches $-\infty$. Therefore, it is easy to see that  $\inf H(x,\gamma)=\lim_{x\rightarrow -\infty} H(x,\gamma)=2\eta-\xi+2$. Since, in this case, the values $\lim_{x\rightarrow -\infty} H(x,y)=\xi+2\eta(1-y)$ and $\lim_{y\rightarrow-\infty} H(x,y)= \xi-2\eta x $ are greater than $2\eta-\xi+2$, we have $\inf H(x,y)=2\eta-\xi+2$. Suppose that $x\geq 0$ and $x+y<2$, then the function $H$ becomes \[ H(x,y)=\frac{\xi(2-y+x)-2\eta x}{2-x-y}. \] The function $H(x,\gamma)$ being an increasing function has minimum given by $H(0,\gamma)=\xi$ and for $0\leq x\leq \gamma$, we have $\lim_{y\rightarrow -\infty}H(x,y)=\xi$. Thus, infimum of the function $H(x,y)$, in this situation, is $\xi$. The part when $x\geq0$ and $x+y>2$, the minimum value of the function $H(x,y)$ occurs at the point $(\gamma,\gamma)$ with $H(\gamma,\gamma)=1/(\gamma-1)$. Choosing minimum amongst the values so obtained, for $\eta+1<\xi<2\eta+1$, we observe that if $\eta\leq 1$, then $\min H(x,y)=2\eta-\xi+2$. If $\eta\geq1$, then \[ \inf H(x,y)=\begin{cases}
2\eta-\xi+2 &\quad\text{if}\quad \eta+1\leq \xi\leq \eta+2\\
1/(\gamma-1) &\quad\text{if}\quad \eta+2\leq \xi\leq 2\eta+1.
\end{cases}  \]

The case $\xi\geq 2\eta+1$ which is equivalent to $\gamma\geq 2$ may be subdivided into six sections depending upon the signs of $2-y$, $x$ and $2-x-y$. For the part $x\leq 0$, $y\leq 2$ and $x+y<2$, the function $H$ may have its infimum as $x$ or $y$ approach $-\infty$. It can be seen that $\lim_{x\rightarrow -\infty} H(x,y)=\xi+2\eta(1-y)\geq \xi-2\eta$ and $\lim_{y\rightarrow-\infty} H(x,y)=\xi-2\eta x\geq \xi$. The case when $x\geq 0$, $y\leq 2$ and $x+y<2$, we note that for $y<2$, $H(0,y)=\xi$ and at the same time, for $0\leq x\leq \gamma$, we have $\lim_{y\rightarrow-\infty} H(x,y)=\xi$. Similarly for $x\geq 0$, $y\leq 2$ and $x+y>2$, it can be verified that $\inf H(\gamma,y)=H(\gamma,2)=\xi-2\eta$. Likewise, the minimum of the function $H(x,y)$ for the case $x\geq 0$ and $y\geq 2$ is $1$. Further, for the part $x\leq 0$, $y\geq 2$ with $x+y>2$, the function $H(x,\gamma)$ has minimum given by $H(0,\gamma)=\xi$. If $x\leq 0$, $y\geq 2$ with $x+y<2$, then $\inf H(x,y)=\xi-2\eta$. The infimum of the function $H(x,y)$ is the least of the values for infimum obtained in different cases above. In this way, we see that the $\min H(x,y)=1$ whenever $\xi>2\eta+1$.

Briefly, we observe that if $\eta\leq 1$, then  \[ \inf_{x,y\leq \gamma} H(x,y)=\begin{cases}
\xi &\quad \text{if}\quad 1 \leq \xi\leq \eta+1\\
2\eta-\xi+2 &\quad \text{if}\quad \eta+1\leq \xi\leq 2\eta+1\\
1 &\quad \text{otherwise.}
\end{cases}  \] But if $\eta\geq 1$, then \[ \inf_{x,y\leq \gamma} H(x,y)=\begin{cases}
1/(1-\gamma) &\quad \text{if}\quad 1\leq \xi\leq \eta\\
\xi &\quad \text{if}\quad \eta \leq \xi\leq \eta+1\\
2\eta-\xi+2 &\quad \text{if}\quad \eta+1\leq \xi\leq \eta+2\\
1/(\gamma-1) &\quad \text{if}\quad \eta+2\leq \xi\leq 2\eta+1\\
1 &\quad \text{otherwise.}
\end{cases}  \]

For the rest of the cases, we follow a similar trend by dividing the subcases into parts in accordance with the sign of $2-y$, $x$ and $2-x-y$. Thus, we derive the minimum of the function $H$ as follows:

\textbf{Case 2: } $x\leq \gamma$ and $\gamma\leq y\leq \rho$.\\
The case if $\eta\leq 1$, we can solve to get \[\inf_{x\leq \gamma, \gamma\leq y\leq \rho} H(x,y)=1.\] Let $1\leq \eta\leq 2$. Then \[ \inf_{x\leq \gamma, \gamma\leq y\leq \rho} H(x,y)=\begin{cases}
1/(1-\gamma) &\quad\text{if}\quad 1\leq \xi\leq \eta \\
2\eta-\xi &\quad\text{if}\quad \eta\leq  \xi\leq 2\eta-1\\
1 &\quad\text{otherwise}
\end{cases}   \] and the case when $\eta\geq 2$, we have \[ \inf_{x\leq \gamma, \gamma\leq y\leq \rho} H(x,y)=\begin{cases}
1/(1-\gamma) &\quad\text{if}\quad 1\leq \xi\leq \eta \\
2\eta-\xi &\quad\text{if}\quad \eta\leq \xi\leq \eta+1\\
(2-\rho+\gamma)/(\gamma+\rho-2) &\quad\text{if}\quad \eta+1\leq \xi\leq 2\eta-1\\
1 &\quad\text{otherwise.}
\end{cases}   \]

\textbf{Case 3: } $x\leq \gamma$ and $y\geq \rho$.\\
For $\eta\leq 1$, we have that the function $H(x,y)$ has infimum given by
\[ \inf_{x\leq \gamma,y\geq \rho} H(x,y)=\begin{cases}
2-\xi &\quad\text{if}\quad 2\eta-1\leq \xi\leq \eta\\
2+\xi-2\eta &\quad\text{if}\quad \eta\leq \xi\leq 1\\
1&\quad\text{otherwise.}
\end{cases}   \]
If we have the condition $1\leq \eta\leq 2$, then \[ \inf_{x\leq \gamma,y\geq \rho} H(x,y)=\begin{cases}
\xi &\quad\text{if}\quad 1\leq \xi\leq  \eta\\
(2-\rho+\gamma)/(\gamma+\rho-2) &\quad\text{if}\quad \eta \leq \xi\leq 2\eta-1\\
1 &\quad\text{otherwise.}
\end{cases}   \] The possibility when $\eta\geq 2$, then it may noted that \[ \inf_{x\leq \gamma,y\geq \rho} H(x,y)=\begin{cases}
(2-\rho+\gamma)/(2-\gamma-\rho) &\quad\text{if}\quad 1\leq \xi\leq \eta-1\\
\xi &\quad\text{if}\quad \eta-1\leq \xi\leq \eta\\
2\eta-\xi &\quad\text{if}\quad \eta \leq \xi\leq \eta+1\\
(2-\rho+\gamma)/(\gamma+\rho-2) &\quad\text{if}\quad \eta+1\leq \xi\leq 2\eta-1\\
1 &\quad\text{otherwise.}
\end{cases}   \]

\textbf{Case 4: } $\gamma\leq x\leq \rho$ and $y\leq \gamma$.\\
As in the cases above, we have infimum of the function $H(x,y)$ to be $\xi$ whenever $1\leq \xi\leq \eta+1$; $(2+\rho-\gamma)/(\gamma+\rho-2)$ whenever $\eta+1\leq \xi\leq 2\eta+1$ and $1$ elsewhere provided $\eta\leq 1$. In case $\eta\geq 1$, then \[ \inf_{\gamma\leq x\leq \rho, y\leq \gamma} H(x,y)=\begin{cases}
1/(1-\gamma) &\quad\text{if}\quad 1\leq \xi\leq \eta\\
\xi &\quad\text{if}\quad \eta\leq \xi\leq \eta+1\\
(2+\rho-\gamma)/(\gamma+\rho-2) &\quad\text{if}\quad \eta+1\leq \xi\leq 2\eta+1\\
1 &\quad\text{otherwise.}
\end{cases}   \]

\textbf{Case 5: } $\gamma\leq x,y \leq \rho$.\\
We compute that for the case if $\eta\leq 1$, $\inf H(x,y)=1$ as $\xi$ ranges over the real line. For the part when $\eta\geq 1$, we have \[\inf_{\gamma\leq x,y\leq \rho} H(x,y)=\begin{cases}
1/(1-\gamma) &\quad\text{if}\quad 1\leq \xi\leq \eta\\
1/(\rho-1) &\quad\text{if}\quad \eta\leq \xi\leq 2\eta-1\\
1 &\quad\text{otherwise.}
\end{cases}\]

\textbf{Case 6: } $\gamma\leq x\leq \rho$ and $y\geq \rho$.\\
Again with condition $\eta\leq 1$, we note that $\inf H(x,y)=1$ for $\xi\in\mathbb{R}$ and the case when $1\leq \eta\leq 2$, then $\inf H(x,y)=\xi$ whenever $1\leq \xi\leq \eta$ and $\inf H(x,y)=1/(\rho-1)$ provided $\eta\leq \xi\leq 2\eta-1$ and the infimum is $1$ elsewhere. On the other hand, if $\eta\geq 2$, then  \[ \inf_{\gamma\leq x \leq \rho, y\geq \rho} H(x,y)=\begin{cases}
(2-\rho+\gamma)(2-\gamma-\rho) &\quad\text{if}\quad 1\leq \xi\leq \eta-1\\
\xi &\quad\text{if}\quad \eta-1\leq \xi\leq \eta\\
1/(\rho-1) &\quad\text{if}\quad \eta\leq \xi\leq 2\eta-1\\
1 &\quad\text{otherwise.}
\end{cases}   \]

\textbf{Case 7: } $x\geq \rho$ and $y\leq \gamma$.\\
In this case, we may see that \[ \inf_{x\geq \rho, y\leq \gamma} H(x,y)=\begin{cases}
(2+\rho-\gamma)/(2-\gamma-\rho) &\quad\text{if}\quad -1\leq \xi\leq \eta-1\\
\xi+2 &\quad\text{if}\quad \eta-1\leq \xi\leq \eta\\
2\eta-\xi+2 &\quad\text{if}\quad \eta\leq \xi\leq \eta+1\\
(2+\rho-\gamma)/(\gamma+\rho-2) &\quad\text{if}\quad \eta+1\leq \xi\leq 2\eta+1\\
1 &\quad\text{otherwise.}
\end{cases}   \]

\textbf{Case 8: } $x\geq \rho$ and $\gamma\leq y\leq \rho$.\\
This case is also partitioned as $\eta\leq 1$ and $\eta\geq 1$. Let $\eta\leq 1$. Then the function $H$ has infimum given by $(2+\rho-\gamma)/(2-\gamma-\rho)$, whenever $-1\leq \xi\leq \eta-1$ and is given by $2\eta-\xi$ if $\eta-1\leq \xi\leq 2\eta-1 $ and is otherwise $1$. If $\eta\geq 1$, then it can be seen that
\[ \inf_{x\geq \rho, \gamma\leq y \leq \rho} H(x,y)=\begin{cases}
(2+\rho-\gamma)/(2-\gamma-\rho) &\quad\text{if}\quad -1\leq \xi\leq \eta-1\\
2\eta-\xi &\quad\text{if}\quad \eta-1\leq \xi\leq \eta\\
1/(\rho-1) &\quad\text{if}\quad \eta\leq \xi\leq 2\eta-1\\
1 &\quad\text{otherwise.}
\end{cases}  \]

\textbf{Case 9: } $x,y\geq \rho$.\\
For this case, if $\eta\leq 1$, then infimum of $H$ happens to be $\xi+2$ if $-1\leq \xi\leq \eta-1$; $2\eta-\xi$ if $\eta-1\leq \xi\leq 2\eta-1$ and $1$ elsewhere. For the part $\eta\geq 1$, we observe \[ \inf_{x,y\geq\rho} H(x,y)=\begin{cases}
1/(1-\rho) &\quad\text{if}\quad -1\leq \xi\leq \eta-2 \\
\xi+2 &\quad\text{if}\quad \eta-2\leq \xi\leq \eta-1 \\
2\eta-\xi  &\quad\text{if}\quad \eta-1\leq \xi\leq \eta \\
1/(\rho-1) &\quad\text{if}\quad \eta\leq \xi\leq 2\eta-1   \\
1 &\quad\text{otherwise.}
\end{cases}  \]
Drawing the conclusion by choosing least of all the infimum values obtained above, the infimum of the function $H$ is determined and is obtained to be \[ \inf_{x,y\in\mathbb{R}} H(x,y)=\begin{cases}
\frac{1}{1-\gamma} &\quad\text{if}\quad 1\leq \xi\leq \eta\\
\frac{1}{\rho-1} &\quad\text{if}\quad \eta\leq \xi\leq 2\eta-1\\
1 &\quad\text{otherwise.}
\end{cases}  \qedhere\]
\end{proof}
Using the above-mentioned lemmas we now prove the theorems stated in Section \ref{intro}.
\begin{proof}[Proof of Theorem \ref{bound theorem}]
Let $f\in\mathcal{R}_\sigma(\lambda,\varphi)$. Using Definition \ref{definition R}, we know that there exist two analytic functions $r$, $s\colon\mathbb{D}\to\mathbb{D}$ satisfying $r(0)=0=s(0)$ such that
\begin{equation}\label{schwarz relation}
(1-\lambda)\frac{f(z)}{z}+\lambda f'(z)=\varphi(r(z))\quad \text{and}\quad (1-\lambda)\frac{g(w)}{w}+\lambda g'(w)=\varphi(s(w)).
\end{equation} Define the functions $p$ and $q$ by
\begin{equation}\label{p and q}
p(z):=\frac{1+r(z)}{1-r(z)}=1+p_1 z+p_2z^2+\cdots \text{ and } q(w):=\frac{1+s(w)}{1-s(w)}=1+q_1 w+q_2w^2+\cdots.
\end{equation} It may be noted that the functions $p$ and $q$ are analytic with positive real part in $\mathbb{D}$ and $p(0)=1=q(0)$. Using equations \eqref{schwarz relation} and \eqref{p and q}, it is clear that \begin{align}
(1-\lambda)\frac{f(z)}{z}+\lambda f'(z)&=\varphi\left(\frac{p(z)-1}{p(z)+1}\right)\label{phi f}\\
\shortintertext{and}
(1-\lambda)\frac{g(w)}{w}+\lambda g'(w)&=\varphi\left(\frac{q(w)-1}{q(w)+1}\right).\label{phi g}
\end{align}  Comparing the coefficients on the both sides of equation \eqref{phi f}, we have the relations
\begin{equation}\label{a2 a3 p}
(1+\lambda)a_2=\frac{B_1p_1}{2} \quad\text{and}\quad (1+2\lambda)a_3=\frac{B_1}{2}p_2+\frac{p_1^2}{4}(B_2-B_1).
\end{equation}
Similarly, using equation \eqref{phi g}, we get
\begin{equation}\label{a2 a3 q}
(1+\lambda)a_2=-\frac{B_1q_1}{2}\quad \text{and}\quad
(1+2\lambda)(2a_2^2-a_3)=\frac{B_1}{2}q_2+\frac{q_1^2}{4}(B_2-B_1).
\end{equation}
Some simple calculations in equation \eqref{a2 a3 p} yield
\begin{equation}\label{fs p}
a_3-xa_2^2=\frac{B_1}{2(1+2\lambda)}\left(p_2-\frac{\nu}{2}p_1^2\right),
\end{equation} where $\nu:=x\frac{B_1}{\tau}-\frac{B_2}{B_1}+1$ and $\tau:=((1+\lambda)^2)/(1+2\lambda)$. From \cite[Lemma 2]{MR2055766}, we have \begin{equation}\label{inequality} |p_2-(\nu/2)p_1^2|\leq \max\{2,2|\nu-1|\} . \end{equation} By means of inequality equation \eqref{fs p} and \eqref{inequality} , we get
\begin{equation}\label{fs x}
|a_3-xa_2^2|\leq \frac{B_1}{1+2\lambda}\max \left\{1, \left|x\frac{B_1}{\tau} -\frac{B_2}{B_1}\right|\right\}.
\end{equation}
With a similar computation using relation \eqref{a2 a3 q}, it follows that
\begin{equation}\label{fs y}
|a_3-(2-y)a_2^2|\leq \frac{B_1}{1+2\lambda}\max \left\{1, \left|y\frac{B_1}{\tau} -\frac{B_2}{B_1}\right|\right\}.
\end{equation}
Using triangle's inequality with the inequalities \eqref{fs x} and \eqref{fs y}, we arrive at \[ |(2-x-y)a_2^2|\leq |a_3-xa_2^2|+|a_3-(2-y)a_2^2|\leq \frac{B_1}{1+2\lambda}(G(x)+G(y)),  \] where $G(x):=\max \left\{1, \left|x\frac{B_1}{\tau} -\frac{B_2}{B_1}\right|\right\}.$ Hence, if $x,y\in\mathbb{R}$, then we have \[ |a_2|^2\leq \frac{B_1}{1+2\lambda}\inf_{x,y\in\mathbb{R}} \frac{G(x)+G(y)}{|2-x-y|}.  \] Since $B_2\in\mathbb{R}$, in view of Lemma \ref{lemma a2} by taking $\eta:=B_1/\tau$ and $\xi:=B_2/B_1$ which leads to $\gamma=\frac{\tau}{B_1}\left(\frac{B_2}{B_1}-1\right)$ and $\rho=\frac{\tau}{B_1}\left(\frac{B_2}{B_1}+1\right)$, we have that $|a_2|$ is bounded by \[ |a_2|\leq \sqrt{\frac{B_1}{1+2\lambda}} \begin{cases}
\frac{B_1}{\sqrt{B_1^2-\tau B_2+\tau B_1}}&\quad\text{if}\quad \frac{B_2}{B_1}\leq \frac{B_1}{\tau}\\
\frac{B_1}{\sqrt{\tau B_2+\tau B_1-B_1^2}}&\quad\text{if}\quad \frac{B_2}{B_1}\geq \frac{B_1}{\tau}.\\
\end{cases}  \] In order to obtain estimate for the coefficient $a_3$, we again apply triangle's inequality to the relations \eqref{fs x} and \eqref{fs y} and infer \[ |a_3| \leq \frac{B_1}{1+2\lambda} \inf_{x,y\in\mathbb{R}} \frac{|2-y|G(x)+|x|G(y)}{|2-x-y|}. \] By Lemma \ref{lemma a3} with $\eta:=B_1/\tau$ and $\xi:=B_2/B_1$, we conclude \[ |a_3|\leq \frac{B_1}{1+2\lambda}\begin{cases}
\frac{B_1^2}{B_1^2-\tau B_2+\tau B_1} &\quad\text{if}\quad 1\leq \frac{B_2}{B_1} \leq \frac{B_1}{\tau}\\
\frac{B_1^2}{\tau B_2+\tau B_1-B_1^2} &\quad\text{if}\quad \frac{B_1}{\tau}\leq \frac{B_2}{B_1}\leq \frac{2B_1}{\tau}-1\\
1 &\quad\text{otherwise}
\end{cases}   \] which completes the proof of the theorem.
\end{proof}
\begin{illustration}\label{illustration R}
Let $\varphi(z)=(1+z)/(1-z)$ and $\lambda=1$.  For $\nu\geq \sqrt{2}$, the function $f_\nu(z):=\nu z/(\nu-z)\in\mathcal{R}_\sigma(1,\varphi)$. The assertion can be justified as follows:\\
The inverse of the function $f_\nu$, denoted by $g_\nu$, is given by $g_\nu(w)=\nu w/(\nu+w)$. Given that $\nu\geq\sqrt{2}$, we can see that $f_\nu$ is univalent and has a univalent inverse in $\mathbb{D}$. For $f_\nu\in\mathcal{R}_\sigma(1,(1+z)/(1-z))$, it is required that for $z,w\in\mathbb{D}$, the subordinations \[ f'_\nu(z) =\frac{\nu^2}{(\nu-z)^2} \prec \frac{1+z}{1-z} \quad\text{and}\quad g'_\nu(w) =\frac{\nu^2}{(\nu+w)^2} \prec \frac{1+w}{1-w}   \] hold. It may be noted that the function $f'_\nu(z)$ maps unit disk $\mathbb{D}$ onto the domain \[ \left|\sqrt{z} -\frac{\nu^2}{\nu^2-1} \right| <\frac{\nu}{\nu^2-1}\] which lies in the right-half plane if and only if $\nu\geq\sqrt{2}$. Similarly, $g'_\nu(w)$ maps $\mathbb{D}$ onto a domain which is contained in the right-half plane provided $\nu\geq\sqrt{2}$. Thus, we have $f_\nu\in\mathcal{R}_\sigma(1,(1+z)/(1-z))$. Moreover, as an application of Theorem \ref{bound theorem}, we must have $1/|\nu|<\sqrt{2/3}$ which is true as $\nu\geq\sqrt{2}$.

Furthermore, the function $f_\nu(z)/z$ maps the unit disk onto the disk \begin{equation}\label{disk}
\left|z -\frac{\nu^2}{\nu^2-1} \right| <\frac{\nu}{\nu^2-1}
\end{equation} which is contained in the right half plane provided $\nu\geq 1$. Similar is the case for the mapping $g_\nu(w)/w$. Therefore, for $\nu\geq 1$ and $z,w\in\mathbb{D}$,  the subordinations \[ \frac{f_\nu(z)}{z}= \frac{\nu}{\nu-z}\prec \frac{1+z}{1-z}\quad \text{and} \quad \frac{g_\nu(w)}{w}=\frac{\nu}{\nu +w}\prec \frac{1+w}{1-w} \] hold. Hence, the function $f_\nu\in \mathcal{R}_\sigma(0,(1+z)/(1-z))$. Theorem \ref{bound theorem} readily yields $1/|\nu|\leq \sqrt{2}$ which clearly is true because of the assumption $\nu\geq 1$.

Now for $\varphi(z)=\sqrt{1+z}$ and $\lambda =0$, the function $f_\nu \in \mathcal{R}_\sigma( 0,\sqrt{1+z})$ if the image of the unit disk under the mappings $f_\nu(z)/z$ and $g_\nu(w)/w$ lie in the region bounded by the right of lemniscate of Bernoulli given by $\{ w:|w^2-1|=1 \}$. By means of \cite[Lemma 2.2]{MR2879136}, it may be noted that if $\nu\geq \sqrt{2}(\sqrt{2}+1)$, then the disk \eqref{disk} is contained in the set $\{w:|w^2-1|<1\}$. Thus, we infer that if $\nu\geq \sqrt{2}(\sqrt{2}+1)$, then $f_\nu\in \mathcal{R}_\sigma(0,\sqrt{1+z})$. Further, using Theorem \ref{bound theorem}, it is easy to compute that $1/|\nu|\leq 1/\sqrt{7}$ which is true as $\nu\geq \sqrt{2}(\sqrt{2}+1)$.
\end{illustration}
Based on the proof of Lemma \ref{lemma a3}, we have the following lemma.
\begin{lemma}\label{lemma starlike}
Let $\xi\in\mathbb{R}$ and $\eta>0$. Let the function $G\colon \mathbb{R}\to[0,\infty)$ be defined as in Lemma \ref{lemma a2}. Then \[ \inf_{x,y\in\mathbb{R}} \frac{|3-y|G(x)+|x+1|G(y)}{|2-x-y|}=\begin{cases}
\frac{2}{1-\gamma} &\quad\text{if}\quad 1-\eta \leq \xi\leq \eta \\
\frac{2}{\rho-1} &\quad\text{if}\quad \eta\leq \xi\leq 3\eta-1\\
1 &\quad\text{otherwise,}
\end{cases}   \]
where $\gamma:=(\xi-1)/\eta$ and $\rho:=(\xi+1)/\eta$.
\end{lemma}
\begin{proof}[Proof of Theorem \ref{bound theorem starlike}]
Since the function $f\in\mathcal{S}^*_\sigma(\varphi)$, the Definition \ref{defn bistarlike} states that there exist two Schwarz functions $r$ and $s$ such that
\begin{equation}\label{starlike schwarz relation}
\frac{zf'(z)}{f(z)}=\varphi(r(z)) \quad \text{and}\quad \frac{wg'(w)}{g(w)}=\varphi(s(w)).
\end{equation}
Let the functions $p$ and $q$ be defined by equation \eqref{p and q}. Clearly, the functions $p$ and $q$ are analytic functions in $\mathbb{D}$ with positive real part and $p(0)=1=q(0)$. Therefore, equation \eqref{starlike schwarz relation} and \eqref{p and q} yield
\begin{equation}\label{starlike phi} \frac{zf'(z)}{f(z)}= \varphi\left(\frac{p(z)-1}{p(z)+1}\right) \quad \text{and}\quad \frac{wg'(w)}{g(w)} =\varphi\left(\frac{q(w)-1}{q(w)+1}\right).  \end{equation} Comparing the coefficients on each side of the above two relations, we get
\begin{align*}
a_2=\frac{B_1p_1}{2}, \quad & 2a_3-a_2^2=\frac{B_2p_1^2}{4} +\frac{B_1}{2}\left(p_2-\frac{p_1^2}{2}\right),\\
a_2=-\frac{B_1q_1}{2}, \quad &\text{and}\quad  3a_2^2-2a_3=\frac{B_2q_1^2}{4} +\frac{B_1}{2}\left(q_2-\frac{q_1^2}{2}\right).
\end{align*} A similar computations as that in proof of Theorem \ref{bound theorem} leads to the following inequalities: \begin{equation}\label{ineq starlike}
|2a_3-(x+1)a_2^2|\leq B_1 G(x)\quad\text{and}\quad |2a_3-(3-y)a_2^2|\leq B_1G(y),
\end{equation} where $G(x):=\max\{1,|xB_1-B_2/B_1|\}.$ On computing using triangle's inequality, it is easy to see that \[ |(2-x-y)a_2^2|\leq |a_3-xa_2^2|+|a_3-(2-y)a_2^2|\leq B_1 (G(x)+G(y))  \] which implies  \[ |a_2|^2\leq B_1 \inf_{x,y\in\mathbb{R}}\frac{G(x)+G(y)}{|2-x-y|}.  \] Since $B_2\in\mathbb{R}$, upon taking $\xi=B_2/B_1$ and $\eta=B_1$, Lemma \ref{lemma a2} gives \[ |a_2|\leq \sqrt{\frac{B_1}{1-\gamma}} \quad \left(\text{if}\quad \frac{B_2}{B_1}\leq B_1\right)\quad \text{and}\quad \sqrt{\frac{B_1}{\rho-1}}\quad \left(\text{if}\quad \frac{B_2}{B_1}\geq B_1, \right)  \] where $\gamma=\frac{1}{B_1}\left(\frac{B_2}{B_1}-1\right)$ and $\rho=\frac{1}{B_1}\left(\frac{B_2}{B_1}+1\right)$. Besides, keeping in view the relation \eqref{ineq starlike}, we may solve to get \[ |a_3|\leq \frac{B_1}{2}\inf_{x,y\in\mathbb{R}} \frac{|3-y|G(x)+|x+1|G(y)}{|2-x-y|}.  \] By means of Lemma \ref{lemma starlike} with $\xi=B_2/B_1$ and $\eta=B_1$ again, on simplifying the above relations, we get the desired estimates for the second and third coefficient of a function in class $\mathcal{S}^*_\sigma(\varphi)$.
\end{proof}
\begin{illustration}
Let $\varphi(z)=(1+z)/(1-z)$. For $\nu\geq1$, the function $f_\nu(z):=\nu z/ (\nu-z) \in \mathcal{S}^*_\sigma((1+z)/(1-z))$. The function $f_\nu$ and its inverse, denoted by $g_\nu$, are univalent in $\mathbb{D}$ for $\nu\geq 1$. For $f_\nu\in\mathcal{S}^*_\sigma((1+z)/(1-z))$, the following subordinations must hold: \[ \frac{zf'_\nu(z)}{f_\nu(z)}=\frac{\nu}{\nu-z} \prec \frac{1+z}{1-z} \quad \text{and} \quad \frac{wg'_\nu(w)}{g_\nu(w)} =\frac{\nu}{\nu+w} \prec \frac{1+w}{1-w}.  \] As in Illustration \ref{illustration R}, the functions $zf'_\nu(z)/f_\nu(z)$ and $wg'_\nu(w)/g_\nu(w)$ map the unit disk onto the region contained in the right-half plane if and only if $\nu\geq1$. Hence, $f_\nu\in \mathcal{S}^*_\sigma ((1+z)/(1-z))$ for $\nu \geq 1$. Further, according to Theorem \ref{bound theorem starlike}, it is required that $1/|\nu|<\sqrt{2}$ which is true as $\nu\geq1$.

Assuming $\nu\geq \sqrt{2}(\sqrt{2}+1)$, using \cite[Lemma 2.2]{MR2879136}, we can see that the mappings $zf'_\nu(z)/f_\nu(z)$ and $wg'_\nu(w)/g_\nu(w)$ map the unit disk onto the disks that are contained in the region $\{ w:|w^2-1|<1  \}$. Hence, the function $f_\nu\in \mathcal{S}^*_\sigma(\sqrt{1+z})$. In this case, Theorem \ref{bound theorem starlike} implies that $1/|\nu|\leq 1/\sqrt{7}$ which is true as $\nu\geq \sqrt{2}(\sqrt{2}+1)$.
\end{illustration}
\begin{remark}
It may be noted that with $\varphi(z)=(1+z)/(1-z)$, whenever $ \nu\geq 1$, the function $f_\nu:=\nu z/ (\nu-z)\in \mathcal{S}^*_\sigma(\varphi) $ and $f_\nu \in \mathcal{R}_\sigma (0,\varphi)$ but for $1\leq \nu<\sqrt{2}$, $f_\nu \not \in \mathcal{R}_\sigma (1,\varphi)$.
\end{remark}

\end{document}